\documentclass[11pt]{article}
\usepackage{latexsym,amssymb,amsmath,amsfonts,amsthm}
\usepackage{graphics}
\usepackage{graphicx}
\usepackage{mathrsfs}
\usepackage{subfigure}
\usepackage{bm}
\newcommand{\R}{{\mat R}}

\newcommand{\N}{{\mat N}}
\newcommand{\C}{{\mat C}}
\newcommand{\ds}{\displaystyle}
\newcommand{\no}{\nonumber}
\newcommand{\be}{\begin{eqnarray}}
\newcommand{\ben}{\begin{eqnarray*}}
\newcommand{\en}{\end{eqnarray}}
\newcommand{\enn}{\end{eqnarray*}}

\newcommand{\pa}{\partial}

\newcommand{\ov}{\overline}

\newcommand{\Rt}{{\rm Re}}
\newcommand{\curl}{{\rm curl\,}}

\newcommand{\dive}{{\rm div\,}}
\newcommand{\g}{\gamma}
\newcommand{\G}{\Gamma}

\newcommand{\vep}{\varepsilon}
\newcommand{\Om}{\Omega}

\newcommand{\mat}{\mathbb}
\newcommand{\se}{\setminus}

\newcommand{\tr}{\triangle}

\newtheorem{theorem}{Theorem}[section]

\newtheorem{corollary}[theorem]{Corollary}

\begin{document}
\renewcommand{\theequation}{\arabic{section}.\arabic{equation}}
\title{\bf  Determining conductivity and embedded obstacles from partial boundary measurements}
\author{
Jiaqing Yang\thanks{School of Mathematics and Statistics, Xi'an Jiaotong University,
Xi'an, Shaanxi, 710049, China ({\tt jiaq.yang@xjtu.edu.cn; jiaqingyang@amss.ac.cn})}
}
\date{}

\maketitle


\begin{abstract}

In this paper, we consider an inverse conductivity problem on a bounded domain 
$\Omega\subset\mathbb{R}^n$, $n\geq2$, also
known as Electrical Impedance Tomography (EIT), for the case where unknown impenetrable obstacles are embedded into $\Omega$. We show that a piecewise-constant conductivity function and embedded obstacles
can be simultaneously recovered in terms of the local Dirichlet-to-Neumann map defined on an arbitrary 
small open subset of the boundary of the domain $\Omega$. The method depends on
the well-posedness of a coupled PDE-system constructed for the conductivity equations in the
$H^1$-space and some elementary a priori estimates for Harmonic functions.


\vspace{.2in}
{\bf Keywords:} Inverse conductivity problem, Dirichlet-to-Neumann map, partial data, embedded obstacle.
\end{abstract}

\section{Introduction}
\setcounter{equation}{0}

This paper is concerned with an inverse boundary value problem for the conductivity equation. We aim to recover the electrical conductivity of a bounded body in $\R^n$, $n\geq2$, as well as the possibly embedded  obstacles  by taking local boundary measurements.
Precisely, let $\Omega$ and $D$ be two bounded domains in $\R^n$ with Lipschitz continuous boundaries  such that $\ov{D}\subseteq \Omega$. In the physical situation, $D$ corresponds to the impenetrable obstacle inside $\Omega$ which consists of possibly several physical components satisfying different boundary conditions. Then the conductivity problem can be 
formulated in finding a weak solution $u\in H^1(\Omega\se\ov{D})$ such that 
\be\label{1.1a}
\left\{
 \begin{array}{ll}
\dive(\g \nabla u) = 0,& \text{in}\;\Omega\se\ov{D}\\
\;u = f,                        &\text{on}\;\pa \Omega\\
\;\mathcal{B}u = 0,    & \text{on}\;\pa D
\end{array}
\right.
\en
for any given voltage potential $f\in H^{1/2}(\partial\Omega)$.  In (\ref{1.1a}), the function 
$\g\in L^\infty(\Omega\se\ov{D})$ denotes the electric conductivity satisfying the ellipticity condition $0<c_0\leq\g\leq c_0^{-1}$ a.e. in the domain $\Omega\se\ov{D}$ for some positive constant 
$c_0\in\R_+$. The operator $\mathcal{B}$ indicates the boundary condition, described by 
$\mathcal{B}u: = u$ if $D$ is a sound-soft obstacle and 
$\mathcal{B} u: =\pa_\nu u+{\rm i}\lambda u$ with $\Rt(\lambda)\geq 0$ if $D$ is an imperfect obstacle. Here, $\nu$ denotes the exterior unit normal on the boundary $\pa D$. The
impedance boundary condition can be reduced to a Neumann boundary condition if $\lambda=0$.

Under the conditions, (\ref{1.1a}) can be proved to be well-posed in $H^1(\Omega\se\ov{D})$. Therefore,
we can define the well-defined Dirichlet-to-Neumann (D-N) map by
\be\label{1.2a}
\Lambda_{\g,D,\mathcal{B}} f = \g\frac{\pa u}{\pa\nu}\Big|_{\pa\Omega} \in H^{-1/2}(\pa\Omega).
\en
If measurements are taken only on an open subset of $\partial\Omega$, denoted by $\Gamma$, then
the relevant map in (\ref{1.2a}) is called the local D-N map. In this case, we write the map for  
$\Lambda^\Gamma_{\g,D,\mathcal{B}}$ which indicates the dependence on the measurement boundary 
$\Gamma$. Then,
the inverse conductivity problem is to determine all unknown $\g, D$ and $\mathcal{B}$ by taking full measurements  $\Lambda_{\g,D,\mathcal{B}}$ or partial measurements 
$\Lambda^\Gamma_{\g,D,\mathcal{B}}$.
This kind of problems can find important applications in diverse fields such as geophysical exploration, nondestructive testing and medical imaging. The first mathematical formulation on the inverse problem can date back to A. P. Calder\'{o}n \cite{Ca80}, where the uniqueness was
considered in recovering $\g$ by the knowledge of $\Lambda_{\g,D,\mathcal{B}}$ for the case 
when no obstacles are embedded 
in the domain $\Omega$, i.e., $D=\emptyset$. Substantial progress has been made 
in this direction since then; see fundamental papers like \cite{KV84}, \cite{KV85}, \cite{Na96}, \cite{SU87} in both the two-dimensional and three-dimensional cases.
It was shown in \cite{SU87} that the global uniqueness theorem was proved in determining 
a sufficiently smooth conductivity or potential for dimension $n\geq3$, by taking measurements over the boundary.
This result was later extended in \cite{B96}, \cite{BT03}, \cite{HT13}, \cite{PPU03} and related references quoted there to the cases of less regularity conductivity or potential. Moreover, in dimension $n=2$, the global
uniqueness theorem was also proved in \cite{Bu08}, \cite{Na96} for a smooth conductivity and \cite{AP06} for the relaxed regularity conductivity, by taking full boundary data.
Compared to the full data case, the inverse conductivity problem by the local D-N map
has been dealt with more recently. We refer to \cite{AV05} \cite{BU02}, \cite{Is07}, \cite{IUY10}, \cite{KSU07} and related references quoted there in both three dimensions
and two dimensions.

The works mentioned above are mainly based on constructing suitable complex geometrical optics solutions associated with the conductivity equation or Schr\"{o}dinger equation. With a different version, a local uniqueness result was proved 
\cite{AV05} on the conductivity, by assuming $\g$ is a piecewise-constant function
 in a known finite partition of the domain; see also
\cite{D98} for the direct geophysical setting of the problem.

We remark that in all above works except \cite{IUY10}, the inverse conductivity problems were dealt with  
only for the case without embedded obstacles in the body. In the current work,
we shall study the inverse problem with impenetrable obstacles embedded in 
$\Omega$, which makes the problem become more challenging since unknown embedded obstacles need to be also recovered by full or partial measurements over the boundary. If $D$ is known to be a sound-soft
obstacle in advance, it has been shown in \cite{IUY10} that the conductivity $\g$ can be uniquely recovered in
the domain $\Omega\se\ov{D}$ by assuming the partial data. 
We also mention a related work \cite{YZ17,YZZ13} from Yang {\em et al.} in the field of the inverse scattering, in which the unique recovery of embedded  obstacles and the surrounding inhomogeneous medium described by a piecewise-constant refractive index were considered.
A novel version was proposed in \cite{YZ17,YZZ13} relying on the $L^2$-well-posedness of the interior transmission problem or its modified version for the 
Helmholtz equations in a sufficiently small domain, which were recently extended to other cases such as the inverse fluid-solid interaction problem \cite{QYZ17} and 
the inverse cavity problem \cite{QY17}. In this paper, motivated by the previous works of the author, 
we shall study the inverse conductivity problem by {\em a priori} assuming that the conductivity
is a piecewise-constant function in an unknown finite partition of the domain. Under this assumption,
we are able to prove that the local D-N map defined on an arbitrary open subset of the boundary
can uniquely determine all unknown $\g$, $D$ and $\mathcal{B}$. 
The approach proposed in this work depends only on the well-posedness of a coupled PDE-system
associated with the conductivity equations in the $H^1$ space, which is  
essentially simple and elementary. Moreover, our approach always works whenever in the two dimensional case or a more higher dimensional case.

Next let us outline the underlying main results of this work. We first present some assumptions pertaining to the domain partition. Assume that the bounded domain $\Omega$ consists of a finite number
of disjoint Lipschitz subdomains $\Omega_\ell$, $\ell=1,2,\cdots,N$, such that
$\ov{\Omega}=\bigcup_{1\leq\ell\leq N}\ov{\Omega_\ell}$. And the electrical conductivity $\g$ is considered to be 
constant in $\Omega_\ell$, $\ell=1,2,\cdots,N$ satisfying a natural condition that $c_{\ell_1}\not=c_{\ell_2}$ if
the interaction $\partial\Om_{\ell_1}\cap\partial\Om_{\ell_2}$ contains a non-empty open subset in the ($n-1$)-dimensional manifold of $\R^n$.
Moreover, we also assume that a smooth obstacle
$D$ is embedded inside some subdomain denoted by $\Omega_{\ell_0}$ $(1\leq\ell_0\leq N)$, i.e., 
$\ov{D}\subset\Omega_{\ell_0}$. That is, the material surrounding the obstacle $D$ is assumed to be 
homogeneous. Under the above considerations, the conductivity $\g$ can be described as the following 
function 
\be\label{1.3a}
\g(x) = \sum_{\ell=1}^{N}c_\ell\cdot\chi_{\Omega_\ell}(x), \quad{\rm for\;almost\;every\;}x
\in\Omega\se\ov{D},
\en
where $c_\ell$, $\ell=1,2,\cdots,N$, are unknown real positive constants, and $\chi_{\Omega_\ell}$ denotes the characteristic function in the subdomain $\Omega_\ell$, which is defined by $1$ in $\Omega_\ell$ and $0$ otherwise. 
For the non-empty open subset $\G$ of the boundary $\partial\Omega$, we define the following
partial Cauchy data 
\be\label{1.4a}
\mathcal{C}_{\g,D,\mathcal{B}}:
=\{(f|_{\G},(\Lambda_{\g,D,\mathcal{B}} f) |_{\G}): f\in H^{\frac{1}{2}}(\partial\Omega)\;
{\rm with}\; f=0\;{\rm on}\;\partial\Omega\se\ov{\G}\}
\en
which will be used to recover internal properties of $\Omega$ in the inverse conductivity problem.
For convenience, we use symbols
$\mathcal{C}_{\g_j,D_j,\mathcal{B}_j}$, $j=1,2$, indicating the dependence on two classes of
different conductivities and embedded obstacles. Based on (\ref{1.3a})-(\ref{1.4a}), we now formulate
 our main result in the following theorem.     
\begin{theorem}\label{thm1.1}
If $\mathcal{C}_{\g_1,D_1,\mathcal{B}_1}=\mathcal{C}_{\g_2,D_2,\mathcal{B}_2}$, then $\g_1=\g_2$, $D_1=D_2$ and $\mathcal{B}_1 = \mathcal{B}_2$.
\end{theorem}

The remaining part of the paper is organized as follows. In Sections \ref{sec2} we introduce an 
auxiliary boundary value problem associated with the modified conductivity equation and prove its
well-posedness in a suitable Sobolev space by employing the variational method. 
In Section \ref{sec4} we propose 
a novel technique for the inverse conductivity problem, also known as Calder\'{o}n problem, of determining the conductivity of a 
piecewise-constant function from Cauchy data measured on an arbitrary open subset of the boundary.

\section{A related PDE-model} \label{sec2}
\setcounter{equation}{0}

In this section, we introduce a related PDE-system in $D_0$ described by two modified conductivity equations which are coupled by a pair of Cauchy conditions on the boundary $\partial D_0$:
\be\label{2.1}
\left\{
 \begin{array}{ll}
\dive(a_1(x) \nabla u_1)-b_1 u_1=\rho_1,& \text{in}\;D_0,\\
\dive(a_2(x) \nabla u_2)- b_2 u_2=\rho_2,& \text{in}\; D_0,\\ 
u_1-u_2=f_1,& \text{on}\;\partial D_0,\\
a_1(x)\partial_\nu u_1-a_2(x)\partial_\nu u_2=f_2,&\text{on}\;\partial D_0,\\
\end{array}
\right.
\en
where $D_0$ is a simply connected, bounded domain with the Lipschitz boundary $\partial D_0$, and 
$\rho_1,\rho_2\in L^2(D_0)$, $f_1\in H^{\frac{1}{2}}(\partial D_0)$ and $f_2\in H^{-\frac{1}{2}}(\partial D_0)$. In (\ref{2.1}), $a_j$, $j=1,2$, are assumed to two bounded measured
functions satisfying $\infty>c_2\geq a_j\geq c_1>0$ for almost every $x\in D_0$, and 
$b_j$, $j=1,2$, are also assumed to be two positive constants. 

Under some appropriate assumptions on $a_j,b_j$, we shall next prove that (\ref{2.1}) is well-posed
in the classical $H^1(D_0)$ space for the given data $\rho_1,\rho_2,f_1$ and $f_2$.  We remark that the well-posedness of 
(\ref{2.1}) will play an important role in the inverse problem of determining a piecewise-constant conductivity $\g$. Indeed, as shown in the proof of Theorem \ref{thm1.1} in next section, the regularity of the solutions of (\ref{1.1a}) with a family of special boundary data $f$ having a singularity like $1/|x-z|$ in the three-dimensional case or $1/\ln |x-z|$ in the two-dimensional case
 can be exactly improved uniformly by analyzing the solvability of (\ref{2.1}). This will lead to a contradiction. The detailed discussions can be found in next section.

In order to study the existence of a solution to (\ref{2.1}), we shall adopt the suggestion from \cite{CC06} with 
employing a technical variational approach for the coupled Helmholtz equations. To this end, we introdue the Hilbert space $X_{a_2}$ defined by 
\be\label{2.2}
X_{a_2}:
=\{\bm{v}\in L^2(D_0)^3: \dive(a_2(x)\bm{v})\in L^2(D_0),\;\curl\bm{v}=0\;\;\text{in\;}D_0\}
\en
under the product 
$(\bm{v}_1,\bm{v}_2)_{X_{a_2}}:=(\bm{v}_1,\bm{v}_2)_{L^2(D_0)^3}
+(\dive(a_2\bm{v}_1),\dive(a_2\bm{v}_2))_{L^2(D_0)}$ for $\bm{v}_1,\bm{v}_2\in X_{a_2}$. Obviously, it makes sense due to the assumption on $a_2$.
The norm on $X_{a_2}$ is thus given by 
\ben
\| \bm{v}\|^2_{X_{a_2}}:=\|\bm{v}\|^2_{L^2(D_0)^2}+\|\dive(a_2\bm{v})\|^2_{L^2(D_0)}.
\enn
If $u_2\in H^1(D_0)$ is the solution of the second equation in (\ref{2.1}), we 
let $\bm{u}_2:=\nabla u_2$ in $D_0$. Then, one has $\bm{u}_2\in L^2(D_0)^3$ satisfying  
\be\label{2.3}
\dive(a_2\bm{u}_2)=b_2u_2+\rho_2\in L^2(D_0)\quad\text{and}\quad \curl\bm{u}_2=0\;\;\text{in\;}D_0,
\en
which yields that $\bm{u}_2\in X_{a_2}$. This fact motivates us to find the solution 
$(u_1,\bm{u}_2)$ of (\ref{2.1}) with the second equation replaced by (\ref{2.3}) on a new product space $H^1(D_0)\times X_{a_2}(D_0)$. By the Green's theorems, it first follows that 
\ben
&&\int_{D_0}\dive(a_2\bm{u}_2)\cdot\dive(a_2\bm{v}){\rm d}x
 =\int_{D_0}(b_2u_2+\rho_2)\dive(a_2\bm{v}){\rm d}x \\
&&\qquad\qquad\qquad\quad=\int_{\pa D_0}a_2 b_2(\nu\cdot\bm{v})u_2{\rm d}s(x)
          -\int_{D_0}a_2 b_2(\bm{v}\cdot\bm{u}_2){\rm d}x
        +\int_{D_0}\rho_2\dive(a_2\bm{v}){\rm d}x\\
&&\qquad\qquad\qquad\quad=\int_{\pa D_0}a_2 b_2(\nu\cdot\bm{v})(u_1-f_1){\rm d}s(x)
      -\int_{D_0}a_2 b_2(\bm{v}\cdot\bm{u}_2){\rm d}x 
       +\int_{D_0}\rho_2\dive(a_2\bm{v}){\rm d}x
\enn
and 
\ben
&&\int_{D_0}(\dive(a_1\nabla u_1)-b_1u_1)\varphi{\rm d}x
=\int_{D_0}\rho_1\varphi{\rm d}x\\
&&\qquad\qquad\qquad\qquad\quad=\int_{\pa D_0}a_1(\partial_{\nu}u_1)\varphi{\rm d}s(x)
       -\int_{D_0}a_1\nabla u_1\cdot\nabla \varphi{\rm d}x-\int_{D_0}b_1 u_1\varphi{\rm d}x\\
&&\qquad\qquad\qquad\qquad\quad=\int_{\pa D_0}a_2(\nu\cdot\bm{u}_2+f_2)\varphi{\rm d}s(x)
       -\int_{D_0}a_1\nabla u_1\cdot\nabla \varphi{\rm dx}-\int_{D_0}b_1 u_1\varphi{\rm d}x
\enn
for each $\bm{v}\in X_{a_2}$ and 
 $\varphi\in H^1(D_0)$, where the integral on $\partial D_0$ is understood in the dual sense of 
$\langle\cdot,\cdot \rangle_{H^{-1/2}\times H^{1/2}}$ due to the term 
$\nu\cdot\bm{v}\in H^{-1/2}(\partial D_0)$ for any $\bm{v}\in X_{a_2}$.

Combining above equalities shows that the solution $(u_1,u_2)$ to (\ref{2.1}) 
satisfies following variation formulation 
\be\label{2.4}
\mathbb{A}(u_1,\bm{u}_2;\varphi,\bm{v})=\mathbb{F}(\varphi;\bm{v}), 
\qquad \text{for}\;(\varphi,\bm{v})\in H^1(D_0)\times X_{a_2},
\en
where the sesquilinear form $\mathbb{A}(\cdot;\cdot)$ 
 and the functional $\mathbb{F}(\cdot)$ on $H^1(D_0)\times X_{a_2}(D_0)$ are defined by
\ben
\mathbb{A}(\varphi,\bm{v}):
&=&\int_{D}(\dive(a_2\bm{u}_2)\cdot\dive(a_2\ov{\bm{v}})
     +a_2 b_2\bm{u}_2\cdot\ov{\bm{v}}){\rm d}x\\
&&+\int_{D}(a_1\nabla u_1\cdot\nabla \ov{\varphi}+b_1u_1\ov{\varphi}){\rm d}x
      -\int_{\pa D}a_2 b_2(\nu\cdot\ov{\bm{v}})  u_1{\rm d}s
      -\int_{\pa D}a_2(\nu\cdot\bm{u}_2)\ov{\varphi}{\rm d}s
 \enn
 and
 \ben
\mathbb{F}(\varphi,\bm{v}):
&=&\int_{D}\rho_2\dive(a_2\ov{\bm{v}}){\rm d}x-\int_{D}\rho_1\ov{\varphi}{\rm d}x
       +\int_{\pa D}a_2f_2\ov{\varphi}{\rm d}s-\int_{\pa D}a_2b_2(\nu\cdot\ov{\bm{v}})f_1{\rm d}s,
\enn
respectively.

Conversely, 
on the other hand, we suppose that $(u_1,\bm{u}_2)$ is one solution to (\ref{2.4}). It follows from the condition $\curl\bm{u}_2=0$ that there
exists some function denoted by $u_2$ such that $u_2\in H^1(D_0)$ with $\bm{u}_2=\nabla(u_2+c)$ for $c\in\C$, since $D_0$ is a simply connected, bounded domain.
Following similar arguments to \cite[Theorem 6.5]{CC06}, we can choose some constant $c_0\in\C$
so that $u_2+c_0$ is the solution to (\ref{2.1}). Thus, the solvability of (\ref{2.1}) is equivalent to that of the variational formulation (\ref{2.4}).

\begin{theorem}\label{thm2.1}
Assume $a_j(x),b_j>c_0>0$ for $j=1,2$. Furthermore, if there exists some positive constant $\vep_0>0$ with  $0<\vep_0<1$ such that 
\ben
\vep_0 b_1>\left( \frac{1+b_2}{2}\right)^2\qquad{\rm and}\qquad
\vep_0\cdot\left(\inf_{x\in D_0}\frac{a_1(x)}{a_2(x)}\right)>\left( \frac{1+b_2}{2}\right)^2,
\enn
then the variational formulation $(\ref{2.4})$ has a unique solution in the sense that 
\be\label{2.5}
\|u_1\|_{H^1(D_0)}+\|\textbf{u}_2\|_{X_{a_2}}
\leq C(\|\rho_1\|_{L^2(D_0)}+\|\rho_2\|_{L^2(D_0)}+\|f_1\|_{H^{1/2}(\partial D_0)}
       +\|f_2\|_{H^{-1/2}(\partial D_0)}).
\en
\end{theorem}
\begin{proof}
Using Schwarz's inequality with the trace theorem, it is first concluded that 
\ben
&&|\mathbb{F}(\varphi,\textbf{u}_2)|\\
&\leq&\|\rho_1\|_{L^2(D_0)}\|\varphi\|_{L^2(D_0)}
          +\|\rho_2\|_{L^2(D_0)}\|\dive(a_2\textbf{v})\|_{L^2(D_0)} \\
&&\quad  +C\|\varphi\|_{H^1(D_0)}\|f_2\|_{H^{-1/2}(\partial D_0)}
  +C(\|a_2 \textbf{v}\|_{L^2(D_0)}
         +\|\dive(a_2\textbf{v})\|_{L^2(D_0)})\|f_1\|_{H^{1/2}(\partial D_0)}\\
&\leq& C(\|\rho_1\|_{L^2(D_0)}+\|\rho_2\|_{L^2(D_0)}
               +\|f_1\|_{H^{1/2}(\partial D_0)}
       +\|f_2\|_{H^{-1/2}(\partial D_0)})(\|\varphi\|_{H^1(D_0)}+\|\textbf{u}_2\|_{X_{a_2}}),
\enn
which yields that $\mathbb{F}(\cdot,\cdot)$ defines a bounded linear functional on the product space
$H^1(D_0)\times X_{a_2}$.

Let $(\varphi;\bm{v}):=(u_1,\bm{u}_2)$. Then, it holds that 
\be\no
|\mathbb{A}(u_1,\bm{u}_2;u_1,\bm{u}_2)|
&=&\int_{D_0}|(\dive(a_2\bm{u}_2)|^2+a_2 b_2|\bm{u}_2|^2){\rm d}x
 +\int_{D_0}(a_1|\nabla u_1|^2+b_1|u_1|^2){\rm d}x \\ \label{2.6}
&& -\int_{\pa D_0}a_2(\nu\cdot\bm{u}_2)\ov{u_1}{\rm d}s(x)
 -\int_{\pa D_0}a_2b_2(\nu\cdot\ov{\bm{u}_2}) u_1{\rm d}s(x).
\en
In (\ref{2.6}), for the integrals on $\partial D$, we apply the divergence theorem to obtain 
\be\label{2.7}
\int_{\pa D_0}a_2(\nu\cdot\bm{u}_2)\ov{u_1}{\rm d}s(x)
&=&\int_{D_0}\dive(a_2\bm{u}_2)\ov{u_1}{\rm d}x+\int_{D_0}a_2\bm{u}_2\cdot\nabla \ov{u}_1{\rm d}x.
\en
Taking the equality (\ref{2.7}) into (\ref{2.6}) deduces 
\be\no
|\mathbb{A}(u_1,\bm{u}_2;u_1,\bm{u}_2)|
&\geq&\int_{D_0}(|\dive(a_2\bm{u}_2)|^2+b_1|u_1|^2-(1+b_2)|\dive(a_2\bm{u}_2)||u_1|){\rm d}x \\ \label{2.7a}
&& +\int_{D_0}a_2\left(|\bm{u}_2|^2+\left(\frac{a_1}{a_2}\right)|\nabla u_1|^2-(1+b_2)|\bm{u}_2||\nabla u_1|\right){\rm d}x.       
\en
For $0<\vep_0<1$, by the mean value inequality, it holds 
\be\label{2.7b}
&&(1+b_2)|\bm{u}_2||\nabla u_1|
\leq 
\vep_0|\bm{u}_2|^2+\frac{(1+b_2)^2}{4\vep_0}|\nabla u_1|^2,\\ \label{2.7c}
&&(1+b_2)|\dive(a_2\bm{u}_2)||u_1|
\leq \vep_0|\dive(a_2\bm{u}_2)|^2+\frac{(1+b_2)^2}{4\vep_0}|u_1|^2.
\en
Define three constants $c_4,c_5$ and $c_6$ by 
$c_3:=b_1-(1+b_2)^2/(4\vep_0)$, $c_4:=\inf_{x\in D_0}(a_1/a_2)-(1+b_2)^2/(4\vep_0)$
and $c_5:=\min\{(1-\vep_0), c_0(1-\vep_0), c_3, c_0c_4\}$, respectively.
It follows from the assumptions on $a_j,b_j$ for $j=1,2$, that $c_3, c_4$ and $c_5$ are all positive.
Thus, we conclude from (\ref{2.7a})-(\ref{2.7c}) that 
\ben
|\mathbb{A}(u_1,\bm{u}_2;u_1,\bm{u}_2)|
&\geq & 
c_5\int_{D_0}(|\dive(a_2\bm{u}_2)|^2+|u_1|^2+
|\bm{u}_2|^2+|\nabla u_1|^2){\rm d}x,
\enn
which means that $\mathbb{A}(\cdot;\cdot)$ is coercive on the space 
$H^1(D_0)\times X_{a_2}$. 
Therefore, the existence
of a solution to (\ref{2.4}) with the stable estimate (\ref{2.5}) follows from Lax-Milgram theorem.
\end{proof}

\begin{corollary}\label{thm2.2}
Let $a_j,b_j$ satisfy the conditions in Theorem $\ref{thm2.1}$. Then, the boundary value problem
$(\ref{2.1})$ has a unique solution $(u_1,u_2)\in H^1(D_0)\times H^1(D_0)$ such that 
\be\label{2.8}
\|u_1\|_{H^1(D_0)}+\|u_2\|_{H^1(D_0)}
\leq C(\|\rho_1\|_{L^2(D_0)}+\|\rho_2\|_{L^2(D_0)}+\|f_1\|_{H^{\frac{1}{2}}(\partial D_0)}
       +\|f_2\|_{H^{-\frac{1}{2}}(\partial D_0)}).
\en
\end{corollary}
\begin{proof}
(\ref{2.8}) is a direct consequence of Theorem \ref{thm2.1} by using the equality
$u_2=\dive(\g_2\bm{u}_2)-\rho_2\in L^2(D_0)$ in (\ref{2.1}).
\end{proof}

\section{Proofs of Theorem \ref{thm1.1}}\label{sec4}
\setcounter{equation}{0}

This section is devoted to the proof of Theorem \ref{thm1.1} when the conductivity $\g$ is described as a piecewise constant function (\ref{1.3a}). In this case, we can show that 
both the embedded obstacle $D$ and coefficients $c_j$, $j=1,2,\cdots,N$, in an unknown finite partition of $\Omega\se\ov{D}$ can be uniquely 
recovered by local boundary measurements $\mathcal{C}_{\g,D,\mathcal{B}}$ on an non-empty open subset $\G$ of $\partial \Omega$. A novel and elementary proof 
will be presented which depends only on the well-posedness of the boundary value problem
(\ref{2.1}).\\

{\bf Proofs of Theorem \ref{thm1.1}.}
For $j=1,2$, let $S_j:=\{\Omega_{j,\ell}: \ell=1,2,\cdots,N_j\}$ denote the sets consisting of a finite
number of subdomains of $\Omega$ with respect to two classes of different parameters 
$(\g_1,D_1,\mathcal{B}_1)$ and $(\g_2,D_2,\mathcal{B}_2)$. Without of loss of generality, we may assume that 
the boundary of $\Omega_{1,1}$ contains a non-empty open subset of 
the measurement boundary $\G$. By the assumption on the obstacle $D_j$, $j=1,2$,
there exists two real number $1\leq\kappa_1\leq N_1$ and $1\leq\kappa_2\leq N_2$ such that 
$D_1\subseteq\Omega_{1,\kappa_1}$ and $D_2\subseteq\Omega_{2,\kappa_2}$, respectively.
For convenience, we could still assume $\kappa_1,\kappa_2\geq2$.

In the following, we shall first prove $\g_2=\g_1$ in $\Omega_{1,1}$ by assuming 
$\mathcal{C}_{\g_1,D_1,\mathcal{B}_1}=\mathcal{C}_{\g_2,D_2,\mathcal{B}_2}$.
Since $\Gamma\cap \partial \Om_{1,1}\not=\emptyset$, we can choose a point $x_*\in (\Gamma\cap\partial\Om_{1,1})$, and define the domain 
$\mathcal{O}_\vep:=B_{\vep}(x_*)\cap \Om_{1,1}$ with sufficiently small $\vep>0$ so that 
both $\g_1$ and $\g_2$ are all constants in $\mathcal{O}_\vep$, due to the assumptions on
$\g_j$, $j=1,2$. Here, $B_{\vep}(x_*)$ denotes a ball centered at $x_*$ with radii $\vep>0$.
We will first prove that $\g_2=\g_1$ in $\mathcal{O}_\vep$ by contradiction.
Without loss of generality, assume that $\g_1>\g_2$ in $\mathcal{O}_\vep$.


For $\ell=1,2$, let $u^{(\ell)}_{j}$, $j\in\N_+$, be the solution to (\ref{1.1a}) with the Dirichlet data
$f(x):=f_j(x)=\chi_\delta(x)\Phi_n(x,x_j)$, 
corresponding to the case $(\g,D,\mathcal{B}):=(\g_\ell,D_\ell,\mathcal{B}_\ell)$. 
Here, $\Phi_n(\cdot,x_j)$ denotes the fundamental solution of the Laplace equation 
$-\tr \Phi_n(x,x_j) =\delta_{x_j}(x)$ in $\R^n$, given by 
\ben
\Phi_n(x,x_j)=
\left\{
 \begin{array}{ll}
\ds \;\;\frac{1}{4\pi}\frac{1}{|x-x_j|}, &\text{for}\;n=3,\\[5mm]
\ds 2\pi\ln|x-x_j|,& \text{for}\;n=2,
\end{array}
\right.
\enn
$\chi_\delta$ is a smooth cut-off function which is supported in $B_\delta(x_*)$ and satisfies  
$|\chi_\delta(x)|\leq 1$ for $x\in\R^n$ and $\chi_\delta(x)=1$ for $x\in B_{\delta/2}(x_*)$, 
and the point sequence $x_j$ are defined by 
\ben
x_j: = x_* + \frac{1}{j}\nu(x_*),\qquad j=J_0,J_0+1,J_0+2,\cdots\cdots,
\enn
where $J_0>0$ is chosen such that $x_j\in\R^n\se\ov{\Omega}$ for all $J_0+\N_+$.
Obviously, one has $f_j = 0$ on $\partial\Omega\se\ov{\G}$, if $\delta>0$ is sufficiently small. 
This means that 
$(f_j|_{\G}, u_j^{(1)}|_{\G}) \in \mathcal{C}_{\g_1,D_1,\mathcal{B}_1}$ and 
$(f_j|_{\G}, u_j^{(2)}|_{\G})\in \mathcal{C}_{\g_2,D_2,\mathcal{B}_2}$ for all large $j\geq J_0$.
Thus, we conclude by $\mathcal{C}_{\g_1,D_1,\mathcal{B}_1}=\mathcal{C}_{\g_2,D_2,\mathcal{B}_2}$ that $\Lambda_{\g_1,D_1,\mathcal{B}_1}f_j = \Lambda_{\g_2,D_2,\mathcal{B}_2}f_j$ on $\G$
for all large $j\geq J_0$.. Then, we consider the following boundary value problem
\be\label{5.2}
\left\{
 \begin{array}{ll}
\dive(\g_1 \nabla \Psi_1)-2 \Psi_1=\zeta_1,& \text{in}\;\mathcal{O}_\vep\\
\dive(\g_2 \nabla \Psi_2)- \Psi_2=\zeta_2,& \text{in}\; \mathcal{O}_\vep\\ 
\Psi_1-\Psi_2=\varrho_1, & \text{on}\;\partial \mathcal{O}_\vep\\
\g_1\partial_\nu \Psi_1- \g_2\partial_\nu \Psi_2=\varrho_2,
             &\text{on}\;\partial \mathcal{O}_\vep\\
\end{array}
\right.
\en
where $\zeta_1:=-2u^{(1)}_j$, $\zeta_2: = -u^{(2)}_j$, 
$\varrho_1: = u^{(1)}_j-u^{(2)}_j$ and $\varrho_2:=\g_1u^{(1)}_j-\g_2u^{(2)}_j$.
We intend to prove the well-posedness of (\ref{5.2}) by utilizing Corollary \ref{thm2.2}, if 
$\zeta_1,\zeta_2,\varrho_1$ and $\varrho_2$ can be shown to be bounded uniformly
for $j\geq J_0$ in the corresponding function spaces.

{\bf $\bullet$ The boundedness of $\zeta_1$ and $\zeta_2$.}
Since $\g_1$ is constant in $\mathcal{O}_\delta$, we can thus consider
the following Dirichlet problem
\be\label{5.3}
\Delta \Theta_j = 0\quad{\rm in\;}\Omega\se\ov{D_1},\qquad 
\Theta_j=f_j\quad{\rm in\;}\partial\Omega\quad{\rm and}\quad
\mathcal{B}_1\Theta_j=0\quad{\rm on\;}\partial D_1.
\en
It is noticed that $f_j(\cdot)=\Phi_n(\cdot,x_j)$ on $(B_{\delta/2}(x_*)\cap\partial\Omega)$, which
yields that the difference $\Pi_j(\cdot): = \Theta_j(\cdot)-\Phi_n(\cdot,x_j)$ in 
$\Omega\se\ov{D_1}$ is bounded uniformly for all $j\geq J_0$ under the classical $H^1$-norm. This is because $\Pi_j(\cdot)$
solves the harmonic equation $\Delta \Pi_j = 0$ in $\Omega\se\ov{D_1}$ with the boundary 
data $\Pi_j (\cdot)= \Theta_j(\cdot)-\Phi_n(\cdot;x_j)$ vanishing for $x\in (B_{\delta/2}(x_*)\cap\partial\Omega)$. Thus, the variational method can be applied to prove the above
assertion. By the definition of $\Phi_n(\cdot,x_j)$, one has 
\be\label{5.4}
\|\Theta_j \|_{L^2(\Omega\se\ov{D_1})}
+\|\Theta_j \|_{H^1(\Omega\se\ov{D_1\cup B_\vep(x_*)})}
\leq C,\quad {\rm for\; sufficiently\;}\vep>0,
\en
where $C>0$ is independent of the choice of $j\geq J_0$. The estimate (\ref{5.4})
leads in a similar way to that the functions $u_j^{(1)}$ have the same property 
\be\label{5.5}
\|u_j^{(1)} \|_{L^2(\Omega\se\ov{D_1})}
+\|u_j^{(1)} \|_{H^1(\Omega\se\ov{D_1\cup B_\vep(x_*)})}
\leq C
\en
uniformly for all $j\geq J_0$ by considering the difference $u^{(1)}_j(\cdot)-\Theta_j(\cdot)$
which satisfies the equation $\dive(\g_1\nabla w)=-\dive((\g_1-\g_1|_{\Omega_{1,1}})\nabla\Theta_j)$ in 
the domain $\Omega\se\ov{D_1}$ in the distribution sense with the homogeneous
boundary data $w=0$ on $\partial \Omega$ and $\mathcal{B}_1w=0$ on $\partial D_1$.
Hence, one has from (\ref{5.5}) that $\zeta_1$ is bounded in $L^2(\mathcal{O}_\delta)$ uniformly all $j\geq J_0$. The same assertion also holds for $\zeta_2$ by a similar discussion.

{\bf $\bullet$ The boundedness of $\varrho_1$ and $\varrho_2$.}
First,  by $\Lambda_{\g_1,D_1,\mathcal{B}_1}f_j = \Lambda_{\g_2,D_2,\mathcal{B}_2}f_j$ on $\G$,
it is known that $\g_1u_j^{(1)} = \g_2u_j^{(1)}$ on $\G$. Consequently, 
$\varrho_1=0$ and $\varrho_2=0$ on $\G$. We then construct a smooth function
$\eta\in C^\infty(\R^n)$ satisfying ${\rm supp}\eta\subseteq B_\delta(x_*)$, $|\eta|\leq1$ in $\R^n$ and $\eta=1$ in $B_{\delta/2}(x_*)$, where $\delta>0$ is sufficiently small such that 
$(B_\delta(x_*)\cap\G)\subseteq\G$. 
Define  
\ben
&&\mathbb{H}_1(x): = (1-\eta(x))(u_j^{(1)}(x)-u_j^{(2)}(x)),\quad {\rm in\;} \mathcal{O}_\vep,\\ 
&&\mathbb{H}_2(x): = (1-\eta(x))(\g_1u_j^{(1)}(x)-\g_2u_j^{(2)}(x)),\quad {\rm in\;} \mathcal{O}_\vep.
\enn
It is easy to check that $\varrho_1 = {\rm Tr}_0\mathbb{H}_1$ and 
\be\label{5.6}
\varrho_2 = {\rm Tr}_1\mathbb{H}_2 - (\partial_\nu(1-\eta))(\g_1u_j^{(1)}(x)-\g_2u_j^{(2)}(x)),
\en
where ${\rm Tr}_\ell$, $\ell=1,2$, denote the trace operators defined by 
${\rm Tr}_0 v = v|_{\partial\mathcal{O}_\vep}$ and ${\rm Tr}_1 v = (\partial_\nu v)|_{\partial\mathcal{O}_\vep}$ for a smooth function $v$. 
Due to $1-\eta=0$ in $B_{\delta/2}(x_*)$, it follows from (\ref{5.5}) that 
\be\label{5.7}
\|{\rm Tr}_1\mathbb{H}_2-\varrho_2\|_{H^{-1/2}(\partial\mathcal{O}_\vep)}+
\|\mathbb{H}_1\|_{H^1(\mathcal{O}_\vep)} +\|\mathbb{H}_2\|_{H^1(\mathcal{O}_\vep)} 
\leq C_\eta
\en
for all $j\geq J_0$. Furthermore, using the equations 
$\Delta u^{(\ell)}_j=0$, one has
$\Delta \mathbb{H}_2=\Delta(1-\eta)(\g_1u_j^{(1)}-\g_2u_j^{(2)})
+2\nabla(1-\eta)\cdot\nabla(\g_1u_j^{(1)}-\g_2u_j^{(2)})$ 
which leads to that 
\be\label{5.8}
\|\Delta \mathbb{H}_2\|_{L^2(\mathcal{O}_\vep)} \leq C_\eta.
\en
With the aid of the trace theorems, we arrive from (\ref{5.7})-(\ref{5.8}) at that
\be\label{5.9}
\|\varrho_1\|_{H^{1/2}(\partial\mathcal{O}_\vep)}
+\|\varrho_2\|_{H^{-1/2}(\partial\mathcal{O}_\vep)}
\leq C_\eta,
\en
which finishes the second assertion.

In (\ref{5.2}), let $a_1:=\g_1$, $a_2:=\g_2$, $b_1:=2$ and $b_2:=1$. Obviously, 
$a_1,a_2,b_1,b_2$ satisfy the conditions in Theorem \ref{thm1.1},
due to $\g_1>\g_2$ in $\mathcal{O}_\vep$. It is then shown by Corollary \ref{thm2.2} that
\be\label{5.10}
\sum_{\ell=1}^2\|\Psi_\ell\|_{H^1(\mathcal{O}_\vep)}
\leq C(\|\zeta_1\|_{L^2(\mathcal{O}_\vep)}+\|\zeta_2\|_{L^2(\mathcal{O}_\vep)}
+\|\varrho_1\|_{H^{1/2}(\partial\mathcal{O}_\vep)}
+\|\varrho_2\|_{H^{-1/2}(\partial\mathcal{O}_\vep)}) 
\leq C
\en
uniformly for $j\geq J_0$.
Notice that, for boundary data defined in (\ref{5.2}), $(\Psi_1,\Psi_2):=(u_j^{(1)},u_j^{(1)})$
is the unique solution to (\ref{5.2}). Using the trace theorem again, it follows from the above
equality that 
\ben
+\infty\longleftarrow\|f_j\|_{H^{1/2}(\partial\mathcal{O}_\vep\cap\G)}
\leq\|u_j^{(1)}\|_{H^1(\mathcal{O}_\vep)}\leq C,\quad{\rm as\;}j\to\infty,
\enn
which clearly leads to a contradiction. Hence, $\g_1=\g_2$ in $\mathcal{O}_\vep$.

If $\Om_{1,1}\subseteq \Om_{2,\ell_0}$ for some $j_0\leq N_2$, then one immediately has 
$\g_2=\g_1$ in $\Om_{1,1}$. Otherwise, we shall claim that it must hold 
\be\label{5.1}
\ov{\Omega_{1,1}} &=& \bigcup_{1\leq j\leq M}\ov{\Omega_{2,\ell_j}}\quad {\rm for\;some\;}M\leq N_2,
\en
which means that $\Om_{1,1}$ consists of a finite number of elements in $S_2$.
We shall prove (\ref{5.1}) by contradiction. 
Assume that it does not hold true for (\ref{5.1}). Then, there exists some subdomain denoted by 
$\Om_{2,j_1}$ in $ S_2$ such that 
$(\Om_{1,1}\cap \Om_{2,j_1})\not=\emptyset$, $\Om_{1,1}\nsubseteq \Om_{2,j_1}$ and 
$\Om_{2,j_1}\nsubseteq \Om_{1,1}$. 
So we can choose two subsets $A_1$ and $A_2$ of $\R^n$ satisfying the following conditions:
${\rm i})$  $A_1\subseteq(\Om_{1,1}\cap \Om_{2,j_1})$; ${\rm ii})$ $A_2\subseteq(\Om_{1,2}\cap \Om_{2,j_1})$ with $\emptyset\not=\partial A_2\cap(\partial\Om_{1,1}\cap\partial\Om_{1,2})$ 
having a non-empty open subset, denoted by $\Sigma$, in the low-dimensional manifold of $\R^n$. Here, $\Om_{1,2}$ denotes the element in $S_1$ which shares a non-empty common partial boundary with $\Om_{1,1}$.
We now take a point $p$ at the boundary $\partial A_1\cap \partial\Om_{2,j_1}$. Recalling 
the assumption on the domain $\Om\se\ov{D}_2$, we can always find a bounded curve $\ell_{x_*,p}\subseteq \Om_{1,1}$ with two endpoints $x_*$ and $p$ such that $\ell_{x_*,p}\cap A$ ($A\in S_2$) contains at most one point. Therefore, there exists a finite number of small balls $B_{\vep_1}(p_j)$ ($1\leq j\leq M_1<\infty$) in $\Om_{1,1}$ with $p_j\in\ell_{x_*,p}$, $p_1=x_*$, $p_{M_1}=p$ and sufficiently small
$\vep_1>0$ such that $B_{\vep_1}:=\bigcup_{1\leq j\leq M_1}B_{\vep_1}(p_j)\supset \ell_{x_*,p}$.

Since $\g_1$ is constant in $\Om_{1,1}$, we let $\g_1(x):=c_{1,1}\in\R_+$ in $\Om_{1,1}$. 
We next claim that $\g_2(x)=c_{1,1}$ in $A'_1: = A_1\cap B_{\vep_1}(p)$. Without loss of generality,
it can be assumed that $\g_2(x) =c_{1,1}$ in $B_{\vep_1}\se\ov{\Om_{2,j_1}}$.
For $\ell=1,2$, we now consider the Dirichlet-Green functions $G_\ell(\cdot, y)$, $y\in(\Om\se\ov{D_\ell})$,
which solves the following boundary value problem
 \ben
\left\{
 \begin{array}{ll}
\dive(\g_\ell\nabla G_\ell(\cdot,y)) = -\delta_y(\cdot)& \text{in}\;\Om\se\ov{D_\ell},\\[1.5mm]
 G_\ell(\cdot,y)=0& \text{on}\; \partial \Omega,\\ [1.5mm]
\mathcal{B}_\ell G_\ell(\cdot,y)=0& \text{on}\;\partial D_\ell.
\end{array}
\right.
\enn 
It is easily checked that $G_\ell(\cdot,y)$ has the same singularity with the fundamental solution
$\Phi_n(\cdot,y)$ at $x=y$.
For any $f\in H^{1/2}(\partial\Omega)$ with ${\rm supp}f\subseteq \G_1:=(\Gamma\cap\partial\Om_{1,1})$,
let $u^{(\ell)}$, $\ell=1,2$, be the solutions to (\ref{1.1a}) with the boundary data $f$, corresponding to 
two classes of different parameters $(\g_1,D_1,\mathcal{B}_1)$ and $(\g_1,D_2,\mathcal{B}_2)$.
Using the representation formulation for $u^{(\ell)}$, $\ell=1,2$, yields 
 \be\no
 u^{(\ell)}(x)
 &=&\int_{\partial \Omega}\g_\ell(y)\left(u^{(\ell)}(y)\frac{\partial G_\ell(x,y)}{\partial\nu(y)}
 -\frac{\partial u^{(\ell)}(y)}{\partial\nu(y)}G_\ell(x,y)\right){\rm d}s(y)\\ \label{5.11}
 &=&\int_{\G_1}\frac{\partial G_\ell(x,y)}{\partial\nu(y)}\g_\ell(y) f(y){\rm d}s(y)\qquad\quad {\rm for\;}x\in B_{\vep_1}(p)\se\ov{\Om_{2,j_1}}.
 \en
Recalling $\Lambda_{\g_1,D_1,\mathcal{B}_1}f = \Lambda_{\g_2,D_2,\mathcal{B}_2}f$, the unique continuation principle can be thus applied for the Laplace equation to deduce $u^{(1)}(x)=u^{(2)}(x)$ in 
$\mathcal{O}_\vep \cap B_{\vep_1}(x_*)$. Then, it holds by turns that $u^{(1)}(x)=u^{(2)}(x)$ in
$B_{\vep_1}\se\ov{\Om_{2,j_1}}$.
Thus, it is shown by (\ref{5.11}) in combination with the fact $\g_1(x)=\g_2(x)$ in $\mathcal{O}_\vep$ that 
\be\label{5.12}
\int_{\G_1} \frac{\partial G_1(x,y)}{\partial\nu(y)} f(y) {\rm d}s(y)
=\int_{\G_1} \frac{\partial G_2(x,y)}{\partial\nu(y)}f(y) {\rm ds}(y)
\en
 for $x\in B_{\vep_1}(p)\se\ov{\Om_{2,j_1}}$ and all $f\in H^{1/2}(\partial\Omega)$ with the support in 
 $\G_1$. This means that $\partial_\nu G_1(x,\cdot)=\partial_\nu G_2(x,\cdot)$ on $\G_1$. Moreover,
 one can also concalud from the similar discussion above (\ref{5.12}) that 
  \be\label{5.13}
 G_1(x,y)=G_2(x,y)\qquad {\rm for\;}x, y\in B_{\vep_1}\se\ov{\Om_{2,j_1}} \;{\rm and \;}\; x\neq y.
 \en

We are at a position to show that $\g_2(x) = c_{1,1}$ in $A'_1$, which will be proved by contradiction.
Assume that $\g_2(x)\neq c_{1,1}$ in $A'_1$. Without loss of generality, we may assume 
$\g_1>\g_2$ in $A'_1$. Similar to (\ref{5.2}), 
we instead consider the following boundary value problem
 \be\label{5.14}
\left\{
 \begin{array}{ll}
\dive(\g_1 \nabla \Psi_1)-2 \Psi_1=\zeta_1,& \text{in}\;A'_1\\
\dive(\g_2 \nabla \Psi_2)- \Psi_2=\zeta_2,& \text{in}\; A'_1\\ 
\Psi_1-\Psi_2=\varrho_1, & \text{on}\;\partial A'_1\\
\g_1\partial_\nu \Psi_1- \g_2\partial_\nu \Psi_2=\varrho_2,
             &\text{on}\;\partial A'_1\\
\end{array}
\right.
\en
with $\zeta_1,\zeta_2,\varrho_1$ and $\varrho_2$ defined by
\ben
&&\zeta_1:=-2G_1(\cdot,x_j),\quad\zeta'_2:=-G_2(\cdot,x_j),\\
&&\varrho_1:=G_1(\cdot,x_j)-G_2(\cdot,x_j),\quad
\varrho_2:=\g_1\partial_\nu G_1(\cdot,x_j)-\g_2\partial_\nu G_2(\cdot,x_j),
\enn
respectively, 
where $x_j$ are defined by $x_j:=p+(1/j)\nu(p)\in B_{\vep_1}(p)\se\ov{\Om_{2,j_1}}$ for all large $j>0$.
%

For $\zeta_j, j=1,2,$ it follows from 
the properties of the Dirichlet-Green's function (cf.\cite{Is90}) that there exists a fixed constant $K>0$, depending only on the coefficients of the conductivity equation, such that 
$K^{-1}\Phi_n(x,y) \leq G_1(x,y),G_2(x,y)\leq K\Phi_n(x,y)$ for all $x,y\in A'_1$.
This leads to that $\zeta_1$ and $\zeta_2$ are bounded in $L^2(A'_1)$
uniformly for sufficiently large $j>0$.

Furthermore, by the equality (\ref{5.13}) and transmission conditions on $\partial A'_1\cap\partial\Om_{2,j_1}$, we observe that $\varrho_1=0$ and $\varrho_2=0$ on $\partial A'_1\cap\partial\Om_{2,j_1}$.
Thus, in order to obtain the
uniform boundedness of $\varrho_1$ and $\varrho_2$ in $H^{1/2}(\partial A'_1)$
and $H^{-1/2}(\partial A'_1)$, respectively, 
 it is enough to estimate $G_1(\cdot,x_j)$ and $G_2(\cdot,x_j)$ in the sense of 
$H^1(A'_1\se\ov{B_{\vep'_1}(p)})$ with $\vep'_1=\vep_1/2$ (see the analysis between (\ref{5.6}) and (\ref{5.9})). Obviously, this holds true due to the properties of the Dirichlet-Green's functions.
Corollary \ref{thm2.2} can be now applied to obtain the estimate
\be\label{5.15}
\|\Psi_1\|_{H^1(A'_1)} +\|\Psi_2\|_{H^1(A'_1)} \leq C
\en
uniformly for sufficiently large $j>0$. Noticing that for each $j>0$, 
$(\Psi_1,\Psi_2):=(G_1(\cdot,x_j),G_2(\cdot,x_j))$ is the unique solution to (\ref{5.14}), one has from (\ref{5.15}) that $G_1(\cdot,x_j)$ is bounded in $H^1(A'_1)$ uniformly for all sufficiently $j>0$. 
However, this is contradicted with the fact 
\ben
\|G_1(\cdot,x_j)\|_{H^1(A'_1)}\to\infty\qquad {\rm as \;}j\to\infty.
\enn
Therefore, one has $\g_2(x)=c_{1,1}$ in $A'_1$, which further yields $\g_2(x)=c_{1,1}$ in $\Om_{2,j_1}$, 
since $A'_1\subset A_1\subset (\Om_{1,1}\cap \Om_{2,j_1})$ and $\g_2(x)$ is constant in $\Om_{2,j_1}$.

Based on the above assertion, we further claim that  $\g_1(x)=\g_2(x)$ in $A_2$, a previously chosen subdomain in condition ${\rm ii})$ after (\ref{5.1}), which is sill proved by contradiction.
Suppose that $\g_1\not=\g_2$ in $A_2$. Without loss of generality, we can assume that $\g_1>\g_2$ in
$A_2$. We now take a new point $p_*\in\Sigma$ (see the definition 
in condition ${\rm ii})$), and define the point sequence $y_j: = p_*+(1/j)\nu(p_*)$ for sufficiently large $j>0$ such that $y_j\in(\Om_{1,1}\cap B_{\vep_2}(p_*))$ for sufficiently small $\vep_2>0$.
Following similar discussions as in deriving (\ref{5.13}), we can first prove that 
\be\label{5.16}
G_1(x,y_j)=G_2(x,y_j)\qquad {\rm for\;}x_j\not=x\in \Om_{1,1}\cap B_{\vep_2}(p_*).
\en
 for each large $j>0$,
This motivates us to consider another new boundary value problem in the subdomain $A'_2:=A_2\cap B_{\vep_2}(p_*)$: \be\label{5.17}
\left\{
 \begin{array}{ll}
\dive(\g_1 \nabla \Psi_1')-2 \Psi_1'=\zeta_1',& \text{in}\;A'_2\\
\dive(\g_2 \nabla \Psi_2')- \Psi_2'=\zeta_2',& \text{in}\; A'_2\\ 
\Psi_1'-\Psi_2'=\varrho_1', & \text{on}\;\partial A'_2\\
\g_1\partial_\nu \Psi_1'- \g_2\partial_\nu \Psi_2'=\varrho_2',
             &\text{on}\;\partial A'_2\\
\end{array}
\right.
\en
which is similar to (\ref{5.14}) with $A_1'$ and $\zeta_1,\zeta_2,\varrho_1,\varrho_2$ replaced by
$A_2'$ and 
\ben
&&\zeta'_1:=-2G_1(\cdot,y_j),\quad\zeta'_2:=-G_2(\cdot,y_j),\\
&&\varrho'_1:=G_1(\cdot,y_j)-G_2(\cdot,y_j),\quad
\varrho'_2:=\g_1\partial_\nu G_1(\cdot,y_j)-\g_2\partial_\nu G_2(\cdot,y_j),
\enn
respectively.

Using the properties on the Dirichlet-Green's functions again yields that both $\zeta_1'$ and $\zeta_2'$ are
bounded in $L^2(A_2')$ uniformly for all large $j>0$. Moreover, (\ref{5.16}) in combination with the transmission conditions on $\Sigma$ leads to $\varrho_1'=\varrho_2=0$ on $\partial A_2'\cap\Sigma$.
This makes that $\varrho_1'$ and $\varrho_2'$ are bounded in $H^{1/2}(\partial A'_2)$
and $H^{-1/2}(\partial A'_2)$, respectively, uniformly for all large $j>0$; see the analysis between (\ref{5.14})
and (\ref{5.15}). Corollary \ref{thm2.2} can be again applied to obtain the estimate
\be\label{5.18}
\|\Psi_1'\|_{H^1(A'_2)} +\|\Psi_2'\|_{H^1(A'_2)} \leq C
\en
uniformly for sufficiently large $j>0$. However, this is a contradiction since $(\Psi_1',\Psi_2'):=(G_1(\cdot,y_j),G_2(\cdot,y_j))$ is the unique solution to (\ref{5.17}) and 
\ben
\|G_1(\cdot,y_j)\|_{H^1(A'_2)}\to\infty,\qquad {\rm as \;}j\to\infty.
\enn
Therefore, we arrive at that $\g_1=\g_2$ in $A_2$. Notice that this is impossible since
$\g_2$ has been proved to be $c_{1,1}$ in $A_2$ and $\g_1:=c_{1,2}$, a new constant, is assumed to
be not $c_{1,1}$ (see the conditions between (\ref{1.3a}) and (\ref{1.4a})). Hence, (\ref{5.1}) holds true.

It still needs to be shown that $\g_2=c_{1,1}$ in $\Om_{1,1}$. By (\ref{5.1}), we can easily choose 
some subdomain denoted by $\Om_{2,l_1}$ such that $\Om_{2,l_1}\subseteq\Om_{1,1}$ and
$\G\cap\partial\Om_{2,l_1}\not=\emptyset$. That is, $\partial\Om_{2,l_1}$ and $\G$ share a non-empty 
open subset in a ($n-1$)-dimensional manifold of $\R^n$. By repeating the above process, one can conclude
$\g_2=c_{1,1}$ in $\Om_{2,l_1}$. If $\Om_{1,1}=\Om_{2,l_1}$, then the proof is ended. If not, we can
choose another subdomain denoted by $\Om_{2,l_2}$ such that $\Om_{2,l_2}\subseteq\Om_{1,1}$ 
and $\emptyset\not=\partial\Om_{2,l_1}\cap\partial\Om_{2,l_2}$ containing a non-empty open subset
in ($n-1$)-dimensional manifold of $\R^n$. Again repeating the previous analysis leads to 
$\g_2=c_{1,1}$ in $\Om_{2,l_2}$. If $\Om_{1,1}=(\Om_{2,l_1}\cup\Om_{2,l_2})$, then the proof is ended.
If not, we continue to choose a new subdomain denoted by $\Om_{2,l_3}$ such that 
$\Om_{2,l_3}\subseteq\Om_{1,1}$, and $\emptyset\not=\partial\Om_{2,l_3}\cap\partial\Om_{2,l_2}$
or $\emptyset\not=\partial\Om_{2,l_3}\cap\partial\Om_{2,l_1}$ contains at least a non-empty open subset
in ($n-1$)-dimensional manifold of $\R^n$; see Figure 2 for a geometry setting. The similar arguments
leads to $\g_2=c_{1,1}$ in $\Om_{2,l_3}$ once again. Thus, we can arrive in turn at 
$\g_2=c_{1,1}$ in $\Om_{1,1}$ since $M$ in (\ref{5.1}) is just a bounded positive integer.


Finally, the inverse conductivity problem can be reduced in a similar way to the unique determination of an unknown obstacle
embedded into a known background medium. The standard arguments as in \cite{KP98} implies that
$D_1=D_2$ and $\mathcal{B}_1=\mathcal{B}_2$. The proof is thus complete.

\section*{Acknowledgements}

This work is supported by the NNSF of China under Grant No. 11771349.


\end{document}